\documentclass{amsart}
\usepackage{amsmath, amsfonts, amssymb,amsthm}

\newtheorem{theorem}{Theorem}
\newtheorem{lemma}[theorem]{Lemma}
\newtheorem{cor}[theorem]{Corollary}
\newtheorem{prop}[theorem]{Proposition}

\theoremstyle{definition}

\newtheorem{defi}[theorem]{Definition}
\newtheorem{remark}[theorem]{Remark}

\def\min{\mathop{\mathrm{min}}}
\def\CC{\mathbb C}

\def\PP{\mathbb P}
\def\K{{\bf K}}

 \def\ord{\mathrm{ord}}
  \def\deg{\mathrm{deg\,}}
 
\def\p{\mathfrak p}

\def\noi{{\noindent}}

\def\gen{\mathfrak g}

\let\a\alpha

\let\d\delta
\let\e\epsilon
\let\f\phi

\let\t\tau

\let\D\Delta

\usepackage{ulem}

\begin{document}
\title{Genus one factors of  curves defined by separated variable polynomials}
\author{Ta Thi Hoai An}
\address{Institute of Mathematics, Vietnam Academy of Science and Technology\\
18 Hoang Quoc Viet Road, Cau Giay District \\
10307 Hanoi,  Vietnam}
\email{tthan@math.ac.vn}
\author{Nguyen Thi Ngoc Diep}
\address{Department of Mathematics\\Vinh University\\Vietnam} \email{ngocdiepdhv@gmail.com}

\begin{abstract} We give some sufficient conditions
on complex
polynomials $P$ and $Q$ to assure that the algebraic plane curve $P(x)-Q(y)=0$
has no irreducible component of genus 0 or 1.
Moreover,
if $\deg (P)=\deg (Q)$ and  if both $P$, $Q$ satisfy Hypothesis I
introduced by H.~Fujimoto, our sufficient conditions are necessary.
\end{abstract}

\thanks{2000\ {\it Mathematics Subject Classification.} Primary  14H45
Secondary  14D41 11C08 12E10 11S80 30D25.}
\thanks{Financial support provided 
by  Vietnam's National Foundation for Science and Technology Development (NAFOSTED)}

\baselineskip=16truept \maketitle \pagestyle{myheadings}
\markboth{Ta Thi Hoai An and Nguyen Ngoc Diep}{Genus one factors of  curves }


\section{Introduction }
Give two polynomials $P$ and $Q$  in one variable over a field  $\K $ of characteristic $p\ge 0$, two questions naturally arise:
First, function theorists have found it interesting to ask when
does the equation $P(f)=Q(g)$ have a nontrivial functional solution $(f,g)?$ 
Second, number theorists want to know whether there are finitely many or
infinitely many $\K$-rational solutions to the equation $P(x)=Q(y)$
when $\K$ is a number field, or possibly a global field of positive
characteristic. The two questions are related in certain cases
by theorems of  Faltings and  Picard: When $\K=\CC,$ Picard's theorem
says $P(f)=Q(g)$ has no solutions consisting of non-constant meromorphic
functions $f$ and $g$ when the plane curve $P(x)=Q(y)$ has no
irreducible components of (geometric) genus $0$ or $1.$ Similarly,
Faltings's Theorem says that if the plane curve $P(x)=Q(y)$ has
no irreducible components of (geometric) genus less than two,
then for each number field $\K$ over which $P$ and $Q$ are defined,
there are only finitely many $\K$-rational solutions to $P(x)=Q(y).$

When the degrees of $P$ and $Q$ are relative prime, one knows
by Ehrenfeucht's criterion (\cite{E}, \cite{Tv}) that the plane curve
$P(x)=Q(y)$ is irreducible.
In this case, Ritt's second theorem completely characterizes
when the curve has genus zero (see \cite[pp 40-41]{Sch}),
and Avanzi and Zannier in \cite{AZ1} completely characterize the case of
genus one. In \cite{Fr1}, Fried gave conditions such that the curve
has genus zero when $\gcd(\deg P,\deg Q)\le2$ and also for arbitrary
$d = \gcd(\deg P,\deg Q)$ provided the degrees of $P$ and $Q$  are
larger than some number $N(d) $ and  the curve is irreducible.
Most of results of this type suppose
the irreducibility of the curve, however, when $\gcd(\deg P,\deg Q)> 1,$
the problem of determining the irreducibility of $P(x)-Q(y)$
remains wide open.

We consider now the case of the complex field $\CC.$
In some special cases and under the assumption that
$P$ is indecomposable (that is, $P$
cannot be written as a composition of two polynomials of degree larger than 1),
Tverberg determined in \cite[Ch. 2]{Tv} whether  
$[P(x)-P(y)]/(x-y)$ could contain a
linear or quadratic factor.
Similarly,  Bilu  \cite{B} determined all the pairs of polynomials 
such that $P(x)-Q(y)$ contains a quadratic factor.
Avanzi and Zannier in \cite{AZ} give a nice characterization of when a
curve of the form $P(x)=cP(y)$ has genus at least 1, where $c$ is a
nonzero complex constant.
In the case of polynomials  satisfying Fijimoto's hypothesis I, 
(i.e when restricted to the zero set of its derivative $P',$
the polynomial $P$ is injective), complete characterizations  for when
all the irreducible components of curve $P(x)-Q(y)=0$ have genus at
least 2 have been given in  \cite{AW}, \cite{AWW1}, \cite{CW}, \cite{Pa}  and also \cite{Fuj}.

In this paper, we will give some sufficient conditions that the plane curve
$P(x)=Q(y)$ has no irreducible component of (geometric) genus 0 or 1
for complex polynomials $P$ and $Q,$ not necessarily satisfying 
Fujimoto's hypothesis I.

Henceforth, all polynomials belong to  $\CC[X]$ and all curves we consider are defined in $\PP^2(\CC)$. We denote the coefficients of $P$ and $Q$ by
\begin{align}\label{P and Q}
&P(X)=a_0+a_1X+ \ldots +a_{n_0-1}X^{n_0-1}+a_{n_0}X^{n_0}+a_nX^n, \cr
&Q(X)=b_0+b_1X+ \ldots +b_{m_0-1}X^{m_0-1}+b_{m_0}X^{m_0}+b_mX^m,
\end{align}
where $a_n,$ $a_{n_0}$, $b_{m_0}$ and $b_m$ are non-zero.

Without loss of generality,
 throughout the paper we will assume that $n\geq m$. 

If one of the polynomials $P$ or $Q$ is linear, say $P(x)=ax+b,$  then  $( \frac 1a Q(f)-b,f)$ is a solution of the equation $P(x)=Q(y),$ where   $f$ is any non-constant meromorphic function. Hence, from now on, 
we always assume that both $P$ and $Q$ are not linear polynomials.

The first result is:

\begin{theorem}\label{th2} Let $m=n$ and $n\geq \max\{n_0,m_0\}+4.$
Suppose that $P(x)-Q(y)$ has no linear factor. Then the plane curve 
$P(x)=Q(y)$ has no irreducible component of genus 0 or 1.
\end{theorem}

We will  denote  by
 $\alpha_1, \alpha_2, ..., \alpha_l$ and
 $\beta_1, \beta_2, ..., \beta_h$ the distinct roots of $P'(X)$ and $Q'(X)$,
 respectively. We will use $p_1, p_2, ..., p_l$
 and $q_1, q_2, ..., q_h$ to denote the multiplicities of the roots in
 $P'(X)$ and $Q'(X)$, respectively. Thus,
 $$P'(X)=na_n{(X-\alpha_1)}^{p_1}{(X-\alpha_2)}^{p_2}...{(X-\alpha_l)}^{p_l},$$
 $$Q'(X)=mb_m{(X-\beta_1)}^{q_1}{(X-\beta_2)}^{q_2}...{(X-\beta_h)}^{q_h}.$$

 The polynomial $P(X)$ is said to satisfy \textit{Hypothesis I} if
 $$P(\alpha_i) \neq P(\alpha_j)\ {\rm whenever}\ i\neq j, \  i, j = 1, 2, \ldots, l ,$$
 or in other words $P$ is injective on the roots of $P'$.

 In order to state the  theorems clearly, we need to introduce the
 following notation:

\medskip

\noi {\bf Notation.} \quad We put:
\begin{align*}&A_0:=\{(i, j)\mid 1\leq i\leq l, 1\leq j\leq h, \
P(\alpha_i)=Q(\beta_{j})\},\\
&A_1:=\{(i, j)\mid (i, j)\in A_0,\  p_i>q_j\},\\
&A_2:=\{(i, j)\mid (i, j)\in A_0 ,\
p_i<q_j\}.\end{align*}
  and we put $l_0:=\#A_0$.

The main results are as follows.

 \begin{theorem}\label{th1} Let $P(X)$ and $Q(X)$ satisfy Hypothesis I
and suppose $P(x)-Q(y)$ has no linear factor.
Then, if
$$\displaystyle\sum_{(i,j)\in {A_1}}(p_i-q_j)+\displaystyle\sum_{1\leq i\leq l, (i,j)\notin {A_0}}{p_i}\geq {n-m+3},$$
then the curve
$P(x)-Q(y)=0$ has no irreducible component of genus 0 or 1.
\end{theorem}

 \begin{cor}\label{cor1} With the same conditions as in Theorem~\ref{th1},
then $P(x)-Q(y)$ has no factor of genus 0 or 1 if the following holds
$$\displaystyle\sum_{(i,j)\in {A_2}}(q_j-p_i)+\displaystyle\sum_{1\leq j\leq h, (i,j)\notin {A_0}}{q_j}\geq 3.$$
\end{cor}

When both the polynomials $P$ and $Q$ satisfy Hypothesis I and their
degrees are the same, we are able to give a sufficient and necessary
condition to assure that the curve  has no irreducible components of genus 0 or 1.

\begin{theorem}\label{th3} Let $P$ and $Q$ be polynomials satisfying Hypothesis I and $\deg P=\deg Q$. 
Then the curve
$P(x)-Q(y)$ has no factor of genus 0 or 1 if and only if
after possibly changing indices none of the following hold:
\begin{enumerate}
\item $P(x)-Q(y)$ has a linear factor.
\item $n=2$ or $n=3$.
\item $n=4$ and either there exists at least two $i$ such that $P(\alpha_i)=Q(\beta_i)$  or  there exists only one $i$ such that $P(\alpha_i)=Q(\beta_i)$ and $|p_i-q_i|=2$.
  \item either $n=p_1+1$, $l=1$, $h=2$, $p_1=q_1+1$, $q_2=1$ and 
$P(\alpha_1)=Q(\beta_1)$; or $n=p_1+2$, $h=1$, $l=2$, $q_1=p_1+1$, $p_2=1$ and
$P(\alpha_1)=Q(\beta_1)$.
  \item $l=h=2$, $p_2=q_2=1$, $p_1=q_1$, $n=p_1+2$, and $P(\alpha_1)=Q(\beta_1)$.
\item $n=5$, $l_0=l=h=3$, $p_3=p_2=q_2=q_3=1$,$p_1=q_1= 2$, $P(\alpha_i)=Q(\beta_i)$, for $i=1,2,3$. 
\item $n=5$, $l_0=l=h=2$, $p_i=q_i=2$, $P(\alpha_i)=Q(\beta_i)$, for $i=1,2$.
\end{enumerate}
\end{theorem}
A main technique to prove these results is constructing two non-trivial  regular 1-forms. This method helps to avoid a difficulty of proving irreducibility of the curve.

 {\it Acknowledgments.} A part of this article was written while the first name
author was visiting  Vietnam Institute for Advanced Study in Mathematics.  She would like to thank the Institute for warm hospitality and partial support.

\section{A Key Lemma }

We first recall some notation (for more detail, see \cite[Section 2]{AWW1}). 

Let $F(z_0,z_1,z_2)$ be a homogeneous polynomial of degree $n$
and let $$C = \{[z_0,z_1,z_2] \in {\PP}^2({\CC}) \mid F(z_0,z_1,z_2) = 0\}.$$
By Euler's theorem, for $[z_0,z_1,z_2]\in C$, we have
\begin{align}\label{4}
z_0\frac{\partial F}{\partial z_0}+z_1\frac{\partial F}{\partial z_1}+z_2\frac{\partial F}{\partial z_2} =0.
\end{align}
The equation of the tangent space of $C$ at the point $[z_0,z_1,z_2]\in C$ is
defined by
\begin{align}\label{5}
dx\frac{\partial F}{\partial z_0}+dy\frac{\partial F}{\partial z_1}+dz\frac{\partial F}{\partial z_2} =0.
\end{align}
Then by Cramer's rule, on the curve $C$ we have
\begin{align*}
\frac{\partial F}{\partial z_0}=\frac{W(z_1,z_2)}{W(z_0,z_1)}\frac{\partial F}{\partial z_2},\quad 
\frac{\partial F}{\partial z_1}=\frac{W(z_2,z_0)}{W(z_0,z_1)}\frac{\partial F}{\partial z_2},
\end{align*}
where $W(z_i,z_j)$ denotes the Wronskian of $z_i$ and $z_j$, as
$$W(z_i,z_j):=\begin{vmatrix}z_i&z_j\\dz_i&dz_j
\end{vmatrix}.$$
Therefore
\begin{align}\label{6}\frac{W(z_1,z_2)}{\displaystyle \frac{\partial F}{\partial z_0}}
=\frac{W(z_2,z_0)}{\displaystyle \frac{\partial F}{\partial z_1}}=\frac{W(z_0,z_1)}{\displaystyle \frac{\partial F}{\partial z_2}}.\end{align}

\begin{defi}  Let $C \subset {\PP}^2({\CC})$ be an algebraic curve. A $1$-form $\omega$ on $C$ is said to
be {\it regular} if it is the restriction (more precisely, the
pull-back) of a rational $1$-form on ${\PP}^2({\CC})$ such
that the pole set of $\omega$ does not intersect $C$. A $1$-form is
said to be of {\it Wronskian type} if it is of the form $\displaystyle \frac{R}{S} W(z_i, z_j)$ for some homogeneous polynomials $R$ and $S$ such
that deg $S$ = deg $R + 2.$
\end{defi}

Note that the condition in the above definition 
ensures a well-defined rational 1-form on ${\PP}^2({\CC})$ since
 \begin{align*}\frac{R}{S} W(z_i, z_j) = \frac{z_j^2R}{S}
\frac{W(z_i, z_j)}{z_j^2}.\end{align*} 

A  holomorphic map
\begin{align*}
\f=(\f_0,\f_1,\f_2) : \D_{\e} = \{t \in {\CC} \mid |t| < \e\}
\to C,\quad \varphi(0) = {\mathfrak p} \end{align*}
 is
referred to as a  {\it holomorphic parameterization of $C$ at
${\mathfrak p}$}. A local holomorphic parameterization exists for
sufficiently small $\epsilon$. A rational function $Q$ on the
curve $C$ is represented by $A/B$ where $A$ and $B$ are
homogeneous polynomials in $z_0,z_1,z_2$ such that $B|_C$ is not
identically zero. Thus $Q \circ \phi$ is a well-defined
meromorphic function on $\D_{\e}$ with Laurent expansion
$$Q \circ \phi(t) = \sum_{i = m}^{\infty} a_i t^i, \qquad a_m
\ne 0.$$ The order of $Q \circ \phi$ at $t = 0$ is by definition
$m$ and shall be denoted by 
\begin{align}\ord_{\mathfrak p,
\f} Q = \ord_{t=0} Q(\f(t)).
\end{align} The
function $Q \circ \phi$ is holomorphic if and only if $m \ge 0$.
The rational function $Q$ is regular at ${\mathfrak p}$ if and only if
$Q \circ \phi$ is holomorphic for  all local holomorphic
parameterizations of $C$ at ${\mathfrak p}$. From now, we write $\ord_{\mathfrak p} Q $ instead  of $\ord_{\mathfrak p,
\f} Q $ for some holomorphic parameterization of $C$.

\begin{lemma}[Key Lemma] Let $C$ be a projective curve of degree $n$ in $\PP^2(\CC)$ defined by $F(z_0,z_1,z_2)=0$. Assume that  there is $i\ne j\ne k \in \{0,1,2\}$ and  two well-defined rational  1-forms of Wronskian type
$$\omega_1=\frac{R_1}{S_1} W(z_i, z_j),\quad\text{ and } \omega_2=\frac{R_2}{S_2} W(z_i, z_j)$$
which satisfy the following
\begin{itemize}\item[(i)]   $S_1,S_2$ are factors of ${\displaystyle \frac{\partial F}{\partial z_k}}$. 
\item[(ii)]  $\omega_1$ and $\omega_2$ are $\CC$-linearly independent
on any irreducible component  of the curve $C.$ 
\item[(iii)]  For $i=1,2,$  $\omega_i$ is
regular at every $\mathfrak p\in \mathcal{S}\cap \mathcal{S}_i$,
where $\mathcal S$ is the set of singular points of $C$ and
$ \mathcal{S}_i$ is the zero set of $S_i$, 
\end{itemize}
Then every irreducible component of the curve has genus at least 2.
\end{lemma}

\begin{proof}The rational 1-form $\omega_1$ has possible poles at $(z_0,z_1,z_2)\in\PP^2$ such that $S_1(z_0,z_1,z_2)=0.$ 
By the hypothesis that $S_1$ is a factor of ${\displaystyle \frac{\partial F}{\partial z_k}}$, we can write
$$\frac{\partial F}{\partial z_k}=S_1H_1,
$$
which implies
\begin{align*}\omega_1&=\frac{R_1H_1 W(z_i, z_j)}{S_1H_1}=\frac{R_1H_1 W(z_i, z_j)}{\displaystyle \frac{\partial F}{\partial z_k}}.\end{align*}
Together with (\ref{6}), we have
\begin{align*}\omega_1=\frac{R_1H_1W(z_1,z_2)}{\displaystyle \frac{\partial F}{\partial z_0}}
=\frac{R_1H_1W(z_2,z_0)}{\displaystyle \frac{\partial F}{\partial z_1}}=\frac{R_1H_1W(z_0,z_1)}{\displaystyle \frac{\partial F}{\partial z_2}}.\end{align*}
Hence, $\omega_1$ only has a possible pole at $(z_0,z_1,z_2)\in \mathcal{S}\cap \mathcal{S}_i$, which  is impossible  by the condition (iii). Therefore,  $\omega_1$ is regular on the curve $C$.

Similarly,  $\omega_2$ is regular on the curve $C$.

Together with the condition (ii), on the curve $C$, there are two
regular 1-forms which are independent each irreducible component.
So, they have genus at least 2.
\end{proof}

\begin{remark} 

  Now, let $F(z_0,z_1,z_2)$ be the homogeneous polynomial of degree $n$
obtained by homogenizing $P(x)-Q(y),$ and let $C$ be the curve defined
by $F(z_0,z_1,z_2)=0$ in $\PP^2$. Obviously, the equation 
$P(x)=Q(y)$ has no non-constant meromorphic solution if and only
if the curve $C$ is Brody hyperbolic, meaning
there are no non-constant holomorphic maps from $\CC$ into $C$.
By Picard's theorem, this is equivalent to every irreducible component of
the curve having genus at least 2. Therefore, if the Key lemma holds, then
the equation $P(x)=Q(y)$ has no non-constant meromorphic functions solutions.
\end{remark}

\section{Proof of Theorems \ref{th2}-\ref{th3}}

Recall
 \begin{align*}
&P(x)=a_0+a_1x+ \ldots +a_{n_0-1}x^{n_0-1}+a_{n_0}x^{n_0}+a_nx^n, \cr
&Q(x)=b_0+b_1x+ \ldots +b_{m_0-1}x^{m_0-1}+b_{m_0}x^{m_0}+b_mx^m,
\end{align*}
where $a_n,$ $a_{n_0}$, $b_{m_0}$ and $b_m$ are non-zero and their derivatives are expressed in the forms
$$P'(x)=na_n{(x-\alpha_1)}^{p_1}{(x-\alpha_2)}^{p_2}...{(x-\alpha_l)}^{p_l},$$
 $$Q'(x)=mb_m{(x-\beta_1)}^{q_1}{(x-\beta_2)}^{q_2}...{(x-\beta_h)}^{q_h}.$$
Recall also that we are assuming $n\ge m$.

As in the remark at the end of the last section,
let $F(z_0,z_1,z_2)$ be the homogeneous polynomial of degree $n$
obtained by homogenizing $P(x)-Q(y),$ and let $C$ be the curve defined by
$F(z_0,z_1,z_2)=0$ in $\PP^2$. 

Denote by $P'(z_0,z_2)$ and $Q'(z_1,z_2)$ the homogenization of the
polynomials $P'(x)$ and $Q'(y)$ respectively.
Hence
\begin{align*}
&\frac{\partial F}{\partial z_0}=P'(z_0,z_2)=na_n{(z_0-\alpha_1z_2)}^{p_1}\dots(z_0-\alpha_lz_2)^{p_l},\cr
&\frac{\partial F}{\partial z_1}=z_2^{n-m}Q'(z_1,z_2)=mb_mz_2^{n-m}{(z_1-\beta_1z_2)}^{q_1}\dots(z_1-\beta_hz_2)^{q_h},\cr
&\frac{\partial F}{\partial z_2}=z_2^{n-m'-1}[sz_2^{m'-n_0}z_0^{n_0}+tz_2^{m'-m''}z_1^{m''}+z_2E(z_0,z_1,z_2)]
\end{align*}
where $s$ and $t$ are  constants such that $st\ne0$,  
$E(z_0,z_1,z_2)$ is a homogeneous
polynomial of degree $m'-1$ which can be calculated explicitly and 
depend only  on $P$ and $Q$, and
\begin{align*}&m'=\max\{n_0,m_0\} \text{  if $ n=m$ and $m'= 
\max\{n_0,m\}$  if $ n>m,$}\\
& m''=m_0  \text{   if $n=m$ and $m''=m $ if $ n>m$.}\end{align*}

\begin{lemma}[{\cite[Lemma 4]{AE}}]\label{singular}
The only possible singular points of the projective curve $C$ are $(0:1:0)$ and the $(\alpha_i:\beta_j:1)$ such that $P(\alpha_i)=Q(\beta_j)$, for
$1\leq i \leq l$ and $1\leq j\leq h.$ Moreover, if $n=m$ 
then the curve has no singularity at infinity.
\end{lemma}

\begin{proof}
   Suppose  $a=(a_1,a_2,a_3)$ is a
  singularity, hence  
$${\frac{\partial F}{ \partial 
z_0}(a)=\frac{\partial F}{\partial z_1}(a)=\frac{\partial F}{\partial 
z_2}(a)=0}.$$

   If $a_3=1$ then ${\frac{\partial F}{\partial z_0}(a)=\frac{\partial 
F}{\partial z_1}(a)=0}$ and $P(a_1)=Q(a_2)$. Hence   $a_1=\alpha_i$
and $a_2=\beta_j$ and $P(\alpha_i)=Q(\beta_j)$, for $1\leq i\leq l$ and $1\leq j\leq h$. 

If $a_3=0$
then $\frac{\partial F}{\partial z_0}(a)=na_1^{n-1}=0$ hence   $a_1=0$.

 Now, 
if  $a=(a_1,a_2,0)$ is a
  singularity at infinity and  $n=m$,   then  $\frac{\partial F}{\partial z_0}(a)=na_na_1^{n-1}=0$ and  $\frac{\partial F}{\partial z_1}(a)=nb_na_2^{n-1}=0$ hence
$a_1=a_2=0$,  which is impossible.
  Therefore the curve has no singularity at infinity when $n=m$.\end{proof}

\subsection{Proof of Theorem \ref{th2}}

In Theorem \ref{th2} we consider $P(x)$ and $Q(x)$ to be polynomials of the same degrees. Hence
\begin{align*}
&\frac{\partial F}{\partial z_0}=P'(z_0,z_2)=na_n{(z_0-\alpha_1z_2)}^{p_1}\dots(z_0-\alpha_lz_2)^{p_l},\cr
&\frac{\partial F}{\partial z_1}=Q'(z_1,z_2)=mb_m{(z_1-\beta_1z_2)}^{q_1}\dots(z_1-\beta_hz_2)^{q_h},\cr
&\frac{\partial F}{\partial z_2}=z_2^{n-m'-1}[sz_2^{m'-n_0}z_0^{n_0}+tz_2^{m'-m_0}z_1^{m_0}+z_2E(z_0,z_1,z_2)]
\end{align*}
where  $m'=\max\{n_0,m_0\}$.

\begin{proof}[{Proof of Theorem \ref{th2}.}]
Consider
\begin{align*}\omega_1:=\frac{W(z_0,z_1)}{z_2^{2}},\quad  \text{ and } \omega_2:=\frac{z_0W(z_0,z_1)}{z_2^{3}}.
\end{align*}
They are well-defined rational  1-forms of Wronskian type
and have a possible pole at infinity (i.e at $z_2=0$).  By Lemma~\ref{singular}, when $m=n$ the curve has no singularity at infinity. It follows 
that $\omega_1$ and $\omega_2$ are regular at every singular point. 
It is easy to see from  the hypothesis that $P(x)-Q(y)$ has no linear factor
and that $\omega_1$ and $\omega_2$ are $\CC$-linearly independent on any
irreducible component  of the curve $C.$ 
However, by the hypothesis $n\ge m':=\max\{n_0,m_0\}+4,$ their denominators are factors of $z_2^{n-m'-1}$ and hence also of  $\displaystyle \frac{\partial F}{\partial z_2}$. 
Therefore by the Key Lemma, every irreducible component of the curve $C$ has genus at least 2,
and hence the equation $P(f)=Q(g)$ has no non-constant meromorphic
function solutions.
\end{proof}

\subsection{Proof of Theorem \ref{th1}}

From here, we always assume that the polynomials $P$ and $Q$
satisfy  hypothesis I.

In the proof of Theorem \ref{th1}, we will need the following lemmas.
First, when the polynomials $P$ and $Q$ satisfy
 hypothesis I,
we will give an upper bound on the cardinality of $A_0$.

\begin{lemma}\label{(i,j)}
Let $P(X)$ and $Q(X)$  satisfy Hypothesis I. Then for each $i$,
  $1\leq i\leq l$,   there exists at most one $j$, $1\leq j\leq h$,  such
that $P(\alpha_i)=Q(\beta_j)$.
  Moreover, $l_0\leq \min\{l, h\}.$\end{lemma}

\begin{proof} For each $i, (1\leq i\leq l)$, assume that there exist
$j_1, j_2,\   1\leq j_1, j_2\leq h$,  such that
$P(\alpha_i)=Q(\beta_{j_1})$ and $P(\alpha_i)=Q(\beta_{j_2})$. This
implies that $Q(\beta_{j_1})=Q(\beta_{j_2})$  and hence
$j_1=j_2$ because  $Q$ satisfies  Hypothesis I. Similarly, there
exists at most one $i, (1\leq i\leq l)$ such that
$P(\alpha_i)=Q(\beta_{j})$
  for each $j, (1\leq i\leq h)$. This ends the proof of  Lemma \ref{(i,j)}.
\end{proof}

Recall that we have set:
\begin{align*}&A_0:=\{(i, j)\mid 1\leq i\leq l, 1\leq j\leq h, \
P(\alpha_i)=Q(\beta_{j})\},\\
&A_1:=\{(i, j)\mid (i, j)\in A_0,\  p_i>q_j\},\\
&A_2:=\{(i, j)\mid (i, j)\in A_0 ,\
p_i<q_j\},\end{align*}
  and we put $l_0:=\#A_0$.

By Lemma \ref{(i,j)},
without loss of generality we may assume that 
\begin{align*}&A_0=\{(1,\t(1)), \ldots, (l_0,\t(l_0))\};\\
&A_1=\{(1,\t(1)), \ldots, (l_1,\t(l_1))\},
\end{align*}
which we do from now on.
In what follows, let $L_{i,j}, 1 \le i \ne j \le l_0$, be the linear
form associated to the line passing through the two points
$(\alpha_i,\beta_{\t(i)},1)$ and $(\alpha_j,\beta_{\t(j)},1).$
Note that $L_{i,j}$ is defined by
\begin{align*}
L_{i,j} &:=(z_1-\beta_{\t(j)}z_2)-\frac{ \beta_{\t(i)}-\beta_{\t(j)}}{\a_i-\a_j}
(z_0-\a_j z_2)\\
&=(z_1-\beta_{\t(i)}z_2)-\frac{
\beta_{\t(i)}-\beta_{\t(j)}}{\a_i-\a_j} (z_0-\a_i z_1).\end{align*}

\begin{lemma}\label{order} Let
$\mathfrak p_i = (\a_i, \beta_{\t(i)}, 1) \in C,$ $i=1,\dots, l_0$.
\begin{itemize}\item[(i)]  Assume that $L_{i,j}$, $1 \le i \ne j \le l_0,$ is
not identically zero on any component of $C$. Then,
\begin{align*}
\ord_{\mathfrak p_i}  L_{i,j} \ge \min\{ \ord_{\mathfrak p_i}
(z_0-\a_iz_2), \ord_{\mathfrak p_i} (z_1-\beta_{\t(i)}z_2)\},
\end{align*} 
 for each local
parameterization at $\mathfrak p_i$
and for each
local parameterization  at $\mathfrak p_j$
\begin{align*}
\ord_{\mathfrak p_j}  L_{i,j} &\ge \min\{ \ord_{\mathfrak p_j}
(z_0-\a_jz_2), \ord_{\mathfrak p_j} (z_1-\beta_{\t(j)}z_2)\}.
\end{align*}
\item[(ii)] $
(p_i+1)\,\ord_{\mathfrak p_i} (z_0-\a_iz_2)= (q_{\t(i)}+1)\,
\ord_{\mathfrak p_i} (z_1-\beta_{\t(i)}z_2).$
\item[(iii)] $ \ord_{\mathfrak p_i} W(z_1,z_2)\ge \ord_{\mathfrak p_i} (z_1-\beta_{\t(j)}z_2) -1.$
\end{itemize}
\end{lemma}

\begin{proof} (i) follows directly from the definition of $L_{i,j}$.

(ii)  By  the following expansion of
$P(x)$ and $Q(x)$:
$$
P(x)=P(\a_i)+\sum_{j=p_i+1}^n \nu_{i,j} (x-\a_i)^j, \quad\text{ and } 
Q(x)=Q(\beta_{\t(i)})+\sum_{j=q_i+1}^m \mu_{i,j} (x-\beta_{\t(i)})^j
$$
where $\nu_{i,p_i+1}, \nu_{i, n}, \mu_{i,q_{\t(i)}+1}$ and $\mu_{i, m}$ are non-zero constants.
If $\mathfrak
p_i=(\a_i,\beta_{\t(i)}, 1)\in C$, then $F(z_0,z_1,z_2)$ can be
expressed in terms of $z_0-\a_iz_2$ and $z_1-\beta_{\t(i)}z_2$ as 
\begin{align*}
&F(z_0,z_1,z_2)= \nu_{i,p_i+1}(z_0-\a_iz_2)^{p_i+1}+\{\text{terms in $z_0-\a_iz_2$  of higher degrees }\}\\
  &\quad \quad+\mu_{i,q_{\t(i)}+1}(z_1-\beta_{\t(i)}z_2)^{q_{\t(i)}+1}+\{\text{terms in $z_1-\beta_{\t(i)}z_2$ of higher degrees}\}.
\end{align*}
The lemma is proved by comparing the orders of the term of lowest degree and terms of higher degrees.

(iii) As $z_2 \equiv 1$ on a neighborhood of $\mathfrak p_i$,  

\centerline {$\ord_{\mathfrak p_i} W (z_1,z_2) =  \ord_{\p_i} dz_2 \ge
\ord_{\mathfrak p_i} (z_1 - \beta_{\t(i)}z_2) - 1.$ }
\end{proof}

\begin{lemma}\label{regular} The following assertions hold, on the curve $C$:
\begin{itemize}  \item[(i)] Given $i\in\{l_0+1,\dots, l\}$,  $\displaystyle\eta_1:=\frac{W(z_1,z_2)}{(z_0-\a_iz_2)^{p_i}} 
$ 
is  regular at finite points.
\item[(ii)]  Given $i\in \{1,2,\dots,l_0\}$. Then $\displaystyle
\eta_2:=\frac{(z_1-\beta_{\t(i)}z_2)^{q_{\t(i)}}W(z_1,z_2)}{(z_0-\a_iz_2)^{p_i}} 
$ 
is  regular at  finite points.
 \item[(iii)]  If $|p_i- q_{\t(i)}|\le 2,$ then
$\displaystyle\eta_3:=\frac{(z_1-\beta_{\t(i)}z_2)W(z_1,z_2)}{(z_0-\alpha_iz_2)}$ is regular at $\p_i$,  except when $p_i=1$ and $q_{\t(i)}=3$.
  \item[(iv)] Given  $i,j\in \{1,2,\dots,l_0\}$ and integers $u, v,$
let
$$\zeta_{u,v}:=\frac{L_{i,j}^uW(z_1,z_2)}{(z_0-\a_iz_2)^v}.$$
Then
\begin{itemize}\item[(a)] $\zeta_{u,v}$ is regular at  $\mathfrak{p_i}=(\alpha_i,\beta_{\t(i)},1)$ if $|p_i-q_{\t(i)}|\le 1,$ $u\ge v$
 and $p_i\ge  v.$

Moreover, 
 $\zeta_{2,1}$ is regular at  $\mathfrak{p_i}=(\alpha_i,\beta_{\t(i)},1)$  if $|p_i-q_{\t(i)}|\le 2$.
\item[(b)]
$\zeta_{1,2}$ and $\zeta_{2,3}$ are regular at  $\mathfrak{p_i}=(\alpha_i,\beta_{\t(i)},1)$ if $p_i=q_{\t(i)}+ 1$.
\end{itemize}\end{itemize}
 \end{lemma}

\begin{proof} (i)
At finite points,  $\eta_1$ has possible pole when $z_2=1$ and $z_0=\a_i$. However, by (\ref{6}) 
\begin{align*}\eta_1&=\frac{\prod_{j=1,...,l,j\ne i}(z_0-\a_jz_2)^{p_j}W(z_1,z_2)}{(z_0-\a_1z_2)^{p_1}\dots (z_0-\a_lz_2)^{p_l}}\\
&=\frac{\prod_{j=1,...,l,j\ne i}(z_0-\a_jz_2)^{p_j}W(z_2,z_0)}{(z_0-\beta_1z_2)^{q_1}\dots (z_0-\beta_hz_2)^{q_h}},\end{align*}
hence,  $(\a_i, a,1)$ is a pole if $a=\beta_k$ for some $k=1,...,h$. By the definition of the set $A_0$ and $l_0$, there does not exist any such $a$. We are done for (i).

(ii) Similar to the case (i), at finite points,  $\eta_2$ has possible pole when $z_2=1$,  $z_0=\a_i$ and
\begin{align*}\eta_2&=\frac{(z_1-\beta_{\t(i)}z_2)^{q_{\t(i)}}\prod_{j=1,...,l,j\ne i}(z_0-\a_jz_2)^{p_j}W(z_1,z_2)}{(z_0-\a_1z_2)^{p_1}\dots (z_0-\a_lz_2)^{p_l}}\\
&=\frac{\prod_{j=1,...,l,j\ne i}(z_0-\a_jz_2)^{p_j}W(z_2,z_0)}{\prod_{j=1,...,h,j\ne \t(i)}(z_0-\beta_jz_2)^{q_j}}.\end{align*}
 By the definition of the set $A_0$, if $(\a_i, a,1)$ is a pole then $a=\beta_{\t(i)}$, but the term $(z_1-\beta_{\t(i)}z_2)$ is canceled in the denominator in the second part of the above formula.
We are done for (ii).

(iii) From Lemma~\ref{order}(ii), we have
\begin{align}\label{p}(p_i+1)\ord_{\mathfrak p_i}(z_0-\a_iz_2)=(q_{\t(i)}+1)\ord_{\p_i}(z_1-\beta_{\t(i)}z_2).
\end{align}
We first prove the following claim.

{\bf Claim.} Assume that  $b=\gcd(p_i+1,q_{\t(i)}+1)$. Then $\ord_{\mathfrak p_i}(z_0-\a_iz_2)\ge \displaystyle\frac{q_{\t(i)}+1}b$. Furthermore, if   $q_{\t(i)}=p_i+2$ and $p_i\ge 2$   then  $\ord_{\mathfrak p_i}(z_0-\a_1z_2)\ge  \max\{3,\displaystyle\frac{p_i+3}2\}$.

Indeed, we can write  $p_i+1=bh_1$ and $q_{\t(i)}+1=bh_2,$ where  $h_1$ and $h_2$ are relatively prime.  From (\ref{p}) we have  $ \ord_{\p_i}(z_0-\a_iz_2)\ge\displaystyle \frac{q_{\t(i)}+1}b.$
 If  $q_{\t(i)}=p_i+2$ then  $b=1$ or $b=2$. It follows that
$\ord_{\mathfrak p_i}(z_0-\a_iz_2)\ge\displaystyle \frac{q_{\t(i)}+1}b\ge \max\{3,\displaystyle\frac{p_i+3}2\}$ if $p_i\ge 2$.

We now go back to prove the lemma.

If $p_i\ge q_{\t(i)}$ then the lemma obviously holds because
from (\ref{p}) we have 
$$\ord_{\mathfrak p_i}(z_0-\a_iz_2)\le\ord_{\p_i}(z_1-\beta_{\t(i)}z_2).$$ 
So, we may assume that $p_i<q_{\t(1)}\le p_i+2$.

If  $q_{\t(i)}= p_i+1$ then,  by the above claim,    $ \ord_{\p_i}(z_0-\a_iz_2)\ge q_{\t(i)}+1=p_i+2$. Hence
\begin{align*}\ord_{\p_i}\eta_3&\ge \ord_{\p_1}(z_1-\beta_{\t(i)}z_2)+\ord_{\p_1} W(z_1,z_2)-\ord_{\p_1}(z_0-\a_iz_2)\\
&\ge2\ord_{\p_1}(z_1-\beta_{\t(i)}z_2)-\ord_{\p_1}(z_0-\a_iz_2)-1\\
&\ge \Big(\frac{2(p_i+1)}{q_{\t(i)}+1}-1\Big)\ord_{\p_i}(z_0-\a_iz_2)-1\\
&\ge \frac{p_i}{p_i+2}\ord_{\p_i}(z_0-\a_iz_2)-1\ge 0.
\end{align*}

If   $q_{\t(i)}= p_i+2$ and $p_1\ge2$ then  by the claim, $ \ord_{\p_i}(z_0-\a_1z_2)\ge 3$. Hence 
\begin{align*}\ord_{\p_i}\eta_3&\ge \ord_{\p_i}(z_1-\beta_{\t(i)}z_2) + \ord_{\p_i}W(z_1,z_2)-\ord_{\p_i}(z_0-\a_iz_2)\\
&\ge 2\ord_{\p_i}(z_1-\beta_{\t(i)}z_2)-\ord_{\p_i}(z_0-\a_iz_2)-1\\
&\ge \Big(\frac{2(p_i+1)}{q_{\t(i)}+1}-1\Big)\ord_{\p_i}(z_0-\a_iz_2)-1\\
&\ge \frac{3(p_i-1)}{p_i+3}-1\ge 0
\end{align*}
if $p_i\ge 3.$ If $p_i=2 $, 
 then, by the claim, 
$\ord_{\mathfrak p_i}(z_0-\a_iz_2)\ge\frac {5}{\gcd(3,5)}=5$, hence $\ord_{\p_i}\eta_3\ge0$. 
The assertion (iii) is proved.

(iv)
 If $p_i\le   q_{\t(i)}$ then  $\ord_{\mathfrak p_i}L_{i,j}\ge \ord_{\mathfrak p_i}(z_1-\beta_{\t(i)}z_2)=\displaystyle \frac{(p_i+1)}{q_{\t(i)}+1}\ord_{\p_i}(z_0-\a_1z_2)$.  Therefore,
\begin{align*}\zeta_{u,v}&\ge(u+1)\ord_{\p_i}(z_1-\beta_{\t(1)}z_2) -v.\, \ord_{\p_i}(z_0-\a_2z_2)-1\\
&\ge \Big( \frac{(u+1)(p_i+1)}{q_{\t(i)}+1}-v\Big)\ord_{\p_i}(z_0-\a_1z_2)-1.
\end{align*}

\noindent  If $q_{\t(i)}=p_i+1$ then, by the above claim,  $\ord_{\p_i}(z_0-\a_iz_2)\ge p_i+2.$ Hence, $\ord_{\p_i}\zeta_{2,1}\ge 2p_i\ge 0,$ and
\begin{align*}
&\ord_{\p_i}\zeta_{u,v}\ge {(u-v)(p_i+1)+(p_i-v)}\ge 0 \text{ if $u\ge v$ and $p_i\ge v$.}
\end{align*}

\noindent If $q_{\t(i)}=p_i+2$ then $\ord_{\p_i}(z_0-\a_2z_2)\ge 2$ and we only consider for $\zeta_{2,1}$. We have  
$$\ord_{\p_i}\zeta_{2,1}\ge\displaystyle \frac{4p_i}{p_i+3}-1=\frac{3p_i-3}{p_1+3}\ge 0$$  for every $p_i\ge 1$.
 
  If $p_i>  q_{\t(i)}$, then   $\ord_{\mathfrak p_i}L_{i,j}\ge \ord_{\mathfrak p_i}(z_0-\a_iz_2)$, hence the assertions (a)  obviously holds.
For the assertion (b), we have $p_i=q_{\t(i)}+1$. Therefore, by the claim, we have  $\ord_{\mathfrak p_i}(z_0-\a_iz_2)\ge p_i$ and hence
\begin{align*}\ord_{\p_i}\zeta_{1,2}&\ge \ord_{\mathfrak p_i}(z_0-\a_iz_2)+\ord_\p W(z_1,z_2) -2 \ord_{\mathfrak p_i}(z_0-\a_iz_2)\\
&\ge \ord_{\p_i}(z_1-\beta_{\t(1)}z_2)- \ord_{\mathfrak p_i}(z_0-\a_iz_2)-1\\
&\ge\big( \frac{p_i+1}{p_i}-1\big)\ord_{\mathfrak p_i}(z_0-\a_iz_2)-1\ge 0,\end{align*}
and
\begin{align*}\ord_{\p_i}\zeta_{2,3}&\ge 2\ord_{\mathfrak p_i}(z_0-\a_iz_2)+\ord_\p W(z_1,z_2) -3 \ord_{\mathfrak p_i}(z_0-\a_iz_2)\\
&\ge \ord_{\p_i}(z_1-\beta_{\t(1)}z_2)- \ord_{\mathfrak p_i}(z_0-\a_iz_2)-1\ge 0.\end{align*}
Thus, the lemma~\ref{regular} is proved.
\end{proof}

\begin{proof}[{Proof of Theorem \ref{th1}.}] 
Consider
\begin{align*}\omega_1&:=\frac{z_0z_2^{\sum_{i=1}^{l_1}(p_i-q_{\t(i)})+\sum_{i=l_0+1}^{l}p_i-3}\prod_{i=1}^{l_1}{(z_1-\beta_{\t(i)}z_2)}^{q_{\t(i)}}W(z_1,z_2)}{\prod_{i=1}^{l_1}{(z_0-\alpha_iz_2)}^{p_i}\prod_{i=l_0+1}^l(z_0-\alpha_iz_2)^{p_i}},\\
\omega_2&:=\frac{z_2^{\sum_{i=1}^{l_1}(p_i-q_{\t(i)}) +\sum_{i=l_0+1}^{l}p_i-2}\prod_{i=1}^{l_1}{(z_1-\beta_{\t(i)}z_2)}^{q_{\t(i)}}W(z_1,z_2)}{\prod_{i=1}^{l_1}{(z_0-\alpha_iz_2)}^{p_i}\prod_{i=l_0+1}^l(z_0-\alpha_iz_2)^{p_i}}.\end{align*}
They are well-defined rational  1-forms of Wronskian type and
 are $\CC$-linearly independent on any irreducible component
 of the curve $C$ because of
the hypothesis that $P(x)-Q(y)$ has no linear factor.
However, by the hypothesis
$$\sum_{i=1}^{l_1}(p_i-q_{\t(i)})+\sum_{i=l_0+1}^{l}p_i-3\ge n-m\ge 0,$$
their denominators are factors of  $\displaystyle \frac{\partial F}{\partial z_0}$. 
We will prove they are regular at every singular point of the curve $C$. By Lemma~\ref{regular}(ii), $\omega_1, \omega_2$ are regular at $(\alpha_i, \beta_{\t(i)},1)$ for $i=1,\dots,l_1$. By Lemma~\ref{regular}(i), $\omega_1, \omega_2$ are regular at  finite singular points for any $j=l_0+1,\dots, l$.  We only have to check the regularity of $\omega_1, \omega_2$ at $(0,1,0)$.
By (\ref{6}) and the fact 
$$\displaystyle \frac{\partial F}{\partial z_1}=mb_mz_2^{n-m}{(z_1-\beta_1z_2)}^{q_1}\dots(z_1-\beta_hz_2)^{q_h},$$
we have
\begin{align*}\omega_1&=\frac{z_0z_2^{\sum_{i=1}^{l_1}(p_i-q_{\t(i)})+\sum_{i=l_0+1}^{l}p_i-3}
\prod_{i=1}^{l_1}{(z_1-\beta_{\t(i)}z_2)}^{q_{\t(i)}}
\prod_{i=l_1+1}^{l_0}{(z_0-\a_i z_2)}^{p_i}W(z_1,z_2)}
{\prod_{i=1}^{l}{(z_0-\alpha_iz_2)}^{p_i}}\\
&=\frac{z_0z_2^{\sum_{i=1}^{l_1}(p_i-q_{\t(i)})+\sum_{i=l_0+1}^{l}p_i-3}
\prod_{i=1}^{l_1}{(z_1-\beta_{\t(i)}z_2)}^{q_{\t(i)}}
\prod_{i=l_1+1}^{l_0}{(z_0-\a_i z_2)}^{p_i}W(z_2,z_0)}
{mb_mz_2^{n-m}{(z_1-\beta_1z_2)}^{q_1}\dots(z_1-\beta_hz_2)^{q_h}}\\
&=\frac{z_0z_2^{\sum_{i=1}^{l_1}(p_i-q_{\t(i)})+\sum_{i=l_0+1}^{l}p_i-3-n+m}
\prod_{i=1}^{l_1}{(z_1-\beta_{\t(i)}z_2)}^{q_{\t(i)}}
\prod_{i=l_1+1}^{l_0}{(z_0-\a_i z_2)}^{p_i}W(z_2,z_0)}
{mb_m{(z_1-\beta_1z_2)}^{q_1}\dots(z_1-\beta_hz_2)^{q_h}},
\end{align*}
which implies $(0,1,0)$ can not be a pole of $\omega_1$.
\end{proof}


\subsection{Proof of theorem  \ref{th3} }

In this section,  we always assume $n=m$.


\begin{lemma}\label{big2} Assume that the curve $C$ has no linear component. If  one of the following holds then the curve $C$ cannot have an irreducible
component of genus 0 or 1.
\begin{itemize}\item[(a)] $l_0\ge 2$ and $\sum_{i=1}^{l_1}(p_i-q_i)+\displaystyle\sum_{i=l_0+1}^l{p_i}=2,$  
 \item[(b)]$l_0\ge1$ and $\displaystyle\sum_{i=l_0+1}^l{p_i}=2,$ except the case when $l_0=1$ and  $p_1=1, q_{\t(1)}=3$.
\item[(c)]$l_0\ge 2$  and $l=l_0+1$, except when $l_0=2$, $p_{l_0+1}=1$ and $p_1=p_2=1.$
\end{itemize}\end{lemma}

\begin{proof} By possibly rearranging the indices,
we only have to consider the following cases:

\item{(1)}  $l_0\ge 2$ and $(1,\t(1)), (2,\t(2))\in {A_0}$  such that $p_1-q_{\t(1)}=2. $
\item{(2)}  $l_0\ge 2$ and $(1,\t(1)), (2,\t(2))\in {A_1}$  such that $p_{i}-q_{\t(i)}=1$ with $i=1,2$.
\item{(3)}  $l_0\ge 2$ and  $(1,\t(1))\in {A_1}$  such that $p_{1}-q_{\t(1)}=1$ and $l= l_0+1$ and $|p_{j}-q_{\t(j)}|\le 1$ for every $j=1,...,l_0$.
\item{(4)} $l_0\ge 1$ and $\displaystyle\sum_{i=l_0+1}^l{p_i}=2,$ except when $l_0=1$ and $p_1=1, q_{\t(1)}=3$.
\item{(5)} $l_0\ge 2$ and $l=l_0+1$ and $p_{l_0+1}=1$ and  $0\le q_{\t(i)}-p_i\le1$ with $i=1,2,...,l_0$, except when $l_0=2$, $p_{l_0+1}=1$ and $p_1=p_2=1.$
 
\noindent (Note that in the case 3, if $|p_{j}-q_{\t(j)}|\ge 3$ then we are done because of Theorem~\ref{th1}, if there exists $j$, $(j\in\{1,...,l_0\})$, such that $p_{j}-q_{\t(j)}=2$ then we go back to the case 1, if $ q_{\t(j)}-p_{j}=2$ then proceed as in case 1. Therefore, we could assume $|p_{j}-q_{\t(j)}|\le 1$ for any $j=1,...,l_0$.)

Corresponding to each case, we will  construct two rational 1-forms of Wronskian type which satisfy all the  conditions  of the Key lemma.

 (1)
\begin{align*}\omega_{1,1}&=\frac{(z_1-\beta_{\t(1)}z_2)^{p_1-2}W(z_1,z_2)}{(z_0-\alpha_1z_2)^{p_1}},\\
\omega_{1,2}&=\frac{L_{1,2}^2(z_1-\beta_{\t(1)}z_2)^{p_1-3}W(z_1,z_2)}{(z_0-\alpha_1z_2)^{p_1}(z_0-\alpha_2z_2)}.\end{align*}

(2) \begin{align*}\omega_{2,1}&=\frac{(z_1-\beta_{\t(1)}z_2)^{p_1-1}(z_1-\beta_{\t(2)}z_2)^{p_2-1}W(z_1,z_2)}{(z_0-\alpha_1z_2)^{p_1}(z_0-\alpha_2z_2)^{p_2}},\\
\omega_{2,2}&=\frac{L_{1,2}(z_1-\beta_{\t(1)}z_2)^{p_1-2}(z_1-\beta_{\t(2)}z_2)^{p_2-1}W(z_1,z_2)}{(z_0-\alpha_1z_2)^{p_1}(z_0-\alpha_2z_2)^{p_2}}.\end{align*}

(3) 
\begin{align*}\omega_{3,1}&=\frac{L_{1,2}^{2}W(z_1,z_2)}{(z_0-\alpha_1z_2)^{2}(z_0-\alpha_2z_2)(z_0-\alpha_{l_0+1}z_2)},\\
\omega_{3,2}&=\frac{(z_1-\beta_{\t(1)}z_2)L_{12}W(z_1,z_2)}{(z_0-\alpha_1z_2)^{2}(z_0-\alpha_2z_2)(z_0-\alpha_{l_0+1}z_2)}.\end{align*}

(4) 
\begin{align*}\omega_{4,1}&=\frac{W(z_1,z_2)}{\prod_{i=l_0+1}^l(z_0-\alpha_{i}z_2)},\\
\omega_{4,2}&=\begin{cases}\displaystyle\frac{L_{1,2}^2W(z_1,z_2)}{(z_0-\alpha_1z_2)(z_0-\alpha_2z_2){\prod_{i=l_0+1}^l(z_0-\alpha_{i}z_2)}}& \text{ if } l_0\ge2\\
\displaystyle\frac{(z_1-\beta_{\t(1)}z_2)W(z_1,z_2)}{(z_0-\alpha_1z_2){\prod_{i=l_0+1}^l(z_0-\alpha_{i}z_2)}} &\text{Otherwise,  except when}\\
&\quad  p_1=1, \text{ and } q_{\t(1)}=3.
\end{cases}\end{align*}

(5)  Assume $p_1\ge p_2\ge\dots\ge p_{l_0}.$ Take
$$\omega_{5,1}=\displaystyle\frac{L_{1,2}W(z_1,z_2)}{(z_0-\alpha_1z_2)(z_0-\alpha_2z_2)(z_0-\alpha_{l_0+1}z_2)}$$
and
\begin{align*}\omega_{5,2}&=\begin{cases}\displaystyle\frac{L_{1,2}^2W(z_1,z_2)}{(z_0-\alpha_1z_2)^2(z_0-\alpha_2z_2)(z_0-\alpha_{l_0+1}z_2)}&\,\,\text{ if $p_1\ge 2$}\\
\displaystyle\frac{L_{1,2}L_{1,3}W(z_1,z_2)}{(z_0-\alpha_1z_2)(z_0-\alpha_2z_2)(x-\alpha_3z_2)(z_0-\alpha_{l_0+1}z_2)} &\, \, \text{ if  $p_1=1$ and $l_0\ge 3$}.
\end{cases}\end{align*} 

We will show that the $\omega_{i,j}$'s satisfy condition (iii) in the Key lemma for all $ i=1,2,...,5$ and $j=1,2$. By Lemma~\ref{singular},  the curve $C$ does not have any singular point at infinity, so we only have to prove they are regular at every point $\mathfrak{p_i}=(\alpha_i,\beta_{\t(i)},1),$ $(i=1,\dots,l_0)$ which are zeros of their respective  denominators. 

We now prove that the $\omega_{i,j}$'s are regular at $\mathfrak{p_1}=(\alpha_1,\beta_{\t(1)},1)$.
By Lemma~\ref{order}, 
\begin{align*}&(p_1+1)\ord_{\mathfrak p_1}(z_0-\a_1z_2)=(q_{\t(1)}+1)\ord_{\p_1}(z_1-\beta_{\t(1)}z_2);\\ &\ord_{\mathfrak p_1}W(z_1,z_2)\ge \ord_{\p_1}(z_1-\beta_{\t(1)}z_2) -1;
\end{align*}
 and if $p_1\ge q_{\t(1)}$ then $\ord_{\p_1}L_{1,2}\ge \ord_{\p_1}(z_0-\a_1z_2).$  Hence, 
\begin{align*}\ord_{\p_1}\omega_{1,2}&=\ord_{\p_1}L^2_{1,2}+\ord_{\p_1}(z_1-\beta_{\t(1)}z_2)^{p_1-3}+\ord_{\p_1}W(z_1,z_2)-\ord_{\p_1}(z_0-\a_1z_2)^{p_1}\\
&\ge(p_1-2)\ord_{\p_1}(z_1-\beta_{\t(1)}z_2) -(p_1-2)\ord_{\p_1}(z_0-\a_1z_2)-1\\
&\ge (p_1-2)\Big(\frac{p_1+1}{p_1-1}-1\Big)\ord_{\p_1}(z_0-\a_1z_2)-1\ge  0,
\end{align*}
 because $\ord_{\p_1}(z_0-\a_1z_2)\ge 1$ and   $p_1-q_{\t(1)}=2$ so $p_1\ge3$. Therefore, $\omega_{1,2}$ is  regular at $\p_1$.
\begin{align*}\ord_{\p_1}\omega_{2,2}&=\ord_{\p_1}L_{1,2}+\ord_{\p_1}(z_1-\beta_{\t(1)}z_2)^{p_1-2}+\ord_{\p_1}W(z_1,z_2)-\ord_{\p_1}(z_0-\a_1z_2)^{p_1}\\
&\ge (p_1-1)\Big(\frac{p_1+1}{p_1}-1\Big)\ord_{\p_1}(z_0-\a_1z_2)-1\ge  0,
\end{align*}
 because $\ord_{\p_1}(z_0-\a_1z_2)\ge 1$ and   $p_1-q_{\t(1)}=1$ so $p_1\ge2$. Therefore, $\omega_{2,2}$ is  regular at $\p_1$.

Now,
$\omega_{1,1},\omega_{2,1}$  are regular at  $\mathfrak{p_1}$  by Lemma~\ref{regular}(i).

We know $\omega_{3,1}, \omega_{3,2}$ are regular at $\p_1$ because $p_1>q_{\t(1)},$ and so  
$$\ord_{\mathfrak p_1}(z_0-\a_1z_2)<\ord_{\p_1}(z_1-\beta_{\t(1)}z_2)$$
 and $\ord_{\mathfrak p_1}(z_0-\a_1z_2)<\ord_{\p_1}L_{1,2}$.

We know $\omega_{4,2}$ is regular at $\p_1$ by  Lemma~\ref{regular}(iv,a) if $l_0\ge 2$ and by Lemma~\ref{regular}(iii) otherwise,
except when $l_0=1$,  $p_1=1$ and $q_{\t(1)}=3$.

We know $\omega_{5,1}$ and $\omega_{5,2}$ are  regular at $\p_1$ by  Lemma~\ref{regular}(iv,a).



We know $\omega_{2,1}$, $\omega_{2,2}$ are regular at $\mathfrak{p_2}$  by Lemma~\ref{regular}~(i). 

We know $\omega_{1,2}$, $\omega_{3,1},$ $\omega_{3,2}$, $\omega_{4,2},$ $\omega_{5,1}$ and  $\omega_{5,2}$ are regular at $\p_2$ by  Lemma~\ref{regular}(iv,a).


 By Lemma~\ref{regular}~(ii),  $\omega_{i,j}$'s, with $i=3,4,5$ and $j=1,2$,  have  no pole  on $z_0-\alpha_{l_0+1}=0;$ $\omega_{4,2}$  has  no pole  on $z_0-\alpha_{l_0+2}=0.$  

We are therefore done with showing  that the
$\omega_{i,j}$'s satisfy condition (iii) in the Key lemma.

By the conditions on $p_i$, condition (i) in the Key lemma is 
also satisfied.

 Because of the hypothesis that the curve $C$ does not have any linear
irreducible components,
we have  that $\omega_{i,1}$ and $\omega_{i,2}$ (with $i=2,3,5$) are
$\CC$-linearly independent  on any irreducible component
of the curve $C.$ On the other hand, we were able to construct
$\omega_{1,1}$ and $\omega_{4,1}$ which  are non-trivial regular 1-forms
on  any irreducible component  of the curve $C.$ Hence, every
irreducible component has genus at least 1. If  $\omega_{j,1}$ and
$\omega_{j,2}$ (for $j=1,$ or 4) are $\CC$-linearly dependent
on some irreducible component  of the curve $C,$ then $C$ must
have a quadratic component, which  contradicts the fact that
any irreducible component has genus at least 1.
Thus, condition (ii) in the Key lemma is satisfied.
\end{proof}

Using similar arguments, we also get the following corollary.

\begin{cor}\label{big2c} Assume that the curve $C$ has no linear
component.
If  one of the following holds then the curve $C$ is Brody hyperbolic:
\begin{itemize}\item[(a)] $l_0\ge 2$ and $\sum_{i=l_1+1}^{l_0}(q_{\t(i)}-p_i)+\displaystyle\sum_{i=l_0+1}^h{q_i}=2,$  
 \item[(b)]$l_0\ge1$ and $\displaystyle\sum_{i=l_0+1}^h{q_i}=2,$ except the case when $l_0=1$ and  $q_{\t(1)}=1, p_1=3$.
\item[(c)]$l_0\ge 2$  and $h=l_0+1$, except when $l_0=2$, $q_{l_0+1}=1$ and $q_{\t(1)}=q_{\t(2)}=1.$
\end{itemize}\end{cor}

\begin{defi}Let $R(z_0,z_1,z_2)=0$ be a curve of degree $\deg R$
over
$\CC$.   
Denote by $\d_R$ {\it the deficiency} of the plane curve  $R(z_0,z_1,z_2)=0$
which is
$$
\d_R=\frac 12(\deg R-1)(\deg R-2)-\frac 12\sum_{\mathfrak p}
 m_{\mathfrak p}(m_{\mathfrak p}-1)
$$
where the sum is taken over all points in   $R(z_0,z_1,z_2)=0$ and $m_{\mathfrak p}$
is the multiplicity of $R(z_0,z_1,z_2)=0$ at $\mathfrak p$.
\end{defi}

\begin{prop}\label{bkq} Let $\mathcal{C}$ be a curve in $\PP^2(\CC)$ of degree $n$.
\item (i) If  $\mathcal{C}$ has only one
 singular point and it  is ordinary of multiplicity $\mu$ which is either $n-1$ or $n-2$, then $\mathcal{C}$ is irreducible.
\item (ii) If  $\mathcal{C}$  has only two
 singular points and they are ordinary of multiplicity $n-1$ and 2
respectively, then $\mathcal{C}$ has a linear component.
\end{prop}

\begin{proof}
Let  $\mathcal{C}$ be define by  $F(z_0,z_1,z_2)=0$ and 
let $H(z_0,z_1,z_2)$ be  its proper irreducible factor
 of degree $d$. Then  $F(z_0,z_1,z_2)=H(z_0,z_1,z_2)G(z_0,z_1,z_2)$ for some $G\in \CC[z_0,z_1,z_2]$ and $1\le d< n.$
Clearly,
$G(z_0,z_1,z_2)$ is not divisible by $H(z_0,z_1,z_2)$ because
$F(z_0,z_1,z_2)=0$ has only finitely many singular points. 

(i)
Let $m_H$ be the
multiplicity of the singular point in $H(z_0,z_1,z_2)=0.$
By Bezout's theorem we have
\begin{align}\label{1}
d(n-d)=m_H(\mu-m_H).
\end{align}
Since the multiplicity of the point in the intersection
of these two curves is not bigger than the degree of each curve, it follows that
$$m_H\leq d \,\,{\rm and}\,\,\mu-m_H\leq n-d.$$
Hence, $m_H\leq d\leq n-\mu+m_H,$ where $\mu$ is either $n-1$ or $n-2$. This is impossible if   $1\le d< n$.
Hence,  $F(z_0,z_1, z_2)$ is irreducible, and we are done for (i).

(ii) In this case, the curve has deficiency $\delta_{\mathcal{C}}=\displaystyle\frac{p_1(p_1+1)}2-\frac{p_1(p_1+1)}2-1<0$. Therefore, the curve is reducible and using the above argument for the case $\mu=n-1$, the curve  $H(z_0,z_1,z_2)=0$ has to pass through both of the singular points.
By Bezout's theorem, 
$$(n-d)d=m_H(n-1-m_H)+1,
$$
from which it follows that $d=1$.
Therefore the curve has a linear component, and we are done for (ii).
\end{proof}

The following lemma is a special case of \cite[proposition 6]{AW}.
 For the convenience of the readers, we will recall  here a brief proof.
\begin{lemma}\label{quadratic} 
Let the curve $C=\{F(z_0,z_1,z_2)=0\}$  have only one singular point, say $(\alpha_1,\beta_{\t(1)},1)$ such that  $p_1=3$ and $q_{\t(1)}=1$.  Then the curve $C$ is birational to a
curve
$R(z_0,z_1,z_2)=0$ with only ordinary singularities.  Furthermore,
$$
\d_R= \d_{C}-1. 
$$
\end{lemma}

\begin{proof}We first make a linear transformation which takes the curve
 to an excellent position, and the point $(\a_i,\beta_{t(i)},1)$ to the origin.  Let
\begin{eqnarray*}
&&R_{01}(z_0,z_1,z_2)\\
&=&F(z_0+\alpha_1z_2,z_0+z_1+\beta_{t(1)}z_2,z_2)\\
&=&\nu_1z_0^{4}z_2^{n-4}+ \nu_{2}z_0^{5}z_2^{n-5}
+\dots+z_0^n\\
&&\quad+\mu_1(z_0+z_1)^{2}z_2^{n-2}+\mu_{2}(z_0+z_1)^{3}z_2^{n-3}
+\dots-c(z_0+z_1)^n
\end{eqnarray*}
where $\nu_{i}$'s and  $\mu_{i}$'s are constant. 
We then perform a quadratic transformation 
\begin{eqnarray*}
&&R_{01}(z_1z_2,z_0z_2,z_0z_1)\\
&=&\nu_1z_2^4z_1^nz_0^{n-4}+\nu_{2}z_2^5z_1^{n}z_0^{n-5}+\dots+z_2^nz_1^n\\
&&+\mu_i\mu_1z_2^{2}(z_0+z_1)^{2}(z_0z_1)^{n-2}+\mu_{2}z_2^3(z_0+z_1)^{3}(z_0z_1)^{n-3}
+\dots-cz_2^n(z_0+z_1)^n\\
&=& z_2^2R_{1}(z_0,z_1,z_2),
\end{eqnarray*}
where 
\begin{align*}&R_{1}(z_0,z_1,z_2)=\nu_1z_2^2z_1^nz_0^{n-4}+\nu_{2}z_2^3z_1^{n}z_0^{n-5}+\dots+z_2^{n-2}z_1^n+\\
&\quad\quad+\mu_1(z_0+z_1)^{2}(z_0z_1)^{n-2}+\mu_{2}z_2(z_0+z_1)^{3}(z_0z_1)^{n-3}
+\dots-cz_2^{n-2}(z_0+z_1)^n.
\end{align*}
For points of $F(z_0,z_1,z_2)=0$ outside of the union of 
the 3 exceptional lines  
$\{z_0=0\}$, $\{z_1=0\}$, and $\{z_2=0\}$,
these transformations preserve the  multiplicities and ordinary multiple
points. It is easy to see that the 3 fundamental points (1,0,0), (0,1,0), and
(0,0,1) become ordinary multiple points of 
$R_{1}(z_0,z_1,z_2)=0$ with multiplicities  $n-m_i-1$, $n-m_i-1$, and $n$
respectively.  We also check that the only non-fundamental point  in the
intersection of
$R_{1}(X,Y,Z)$ with the union of three exceptional lines  is $\mathfrak
q_1=(1,-1,0).$  Since
\begin{align*}R_{1}(1,z_1,z_2)=&\nu_1z_2^2z_1^n+\dots+z_2^{n-2}z_1^n+\mu_1(1+z_1)^{2}z_1^{n-2}
+\dots-cz_2^{n-2}(1+z_1)^n,
\end{align*}
the point $\mathfrak q_{1}$ is an ordinary multiple point of multiplicity 2 and
$\d_{R_1}=\d_{F_c}-1$.
\end{proof}

\begin{lemma} \label{dkd}Assume that $P$ and $Q$ are polynomials satisfying  Hypothesis I and $\deg P=\deg Q$. If one of the following holds, then the curve $C$ has an  irreducible component of genus  0 or 1.
\begin{enumerate}
\item $P(X)-Q(Y)$ has a linear factor.
\item $n=2$ or $n=3$.
\item $n=4$ and either there exists at least two $i$ such that $P(\alpha_i)=Q(\beta_{\t(i)})$  or  there exists only one $i$ such that $P(\alpha_i)=Q(\beta_{\t(i)})$ and $|p_i-q_{\t(i)}|=2$.
  \item either $n=p_1+1$, $l=1$, $h=2$, $p_1=q_1+1$, $q_2=1$ and 
$P(\alpha_1)=Q(\beta_1)$; or $n=p_1+2$, $h=1$, $l=2$, $q_1=p_1+1$, $p_2=1$ and
$P(\alpha_1)=Q(\beta_1)$.
  \item $l=h=2$, $p_2=q_2=1$, $p_1=q_1$, $n=p_1+2$, and $P(\alpha_1)=Q(\beta_1)$.
\item $n=5$, $l_0=l=h=3$, $p_3=p_2=q_2=q_3=1$,$p_1=q_{\t(1)}= 2$, $P(\alpha_i)=Q(\beta_{\t(i)})$, for $i=1,2,3$. 
\item $n=5$, $l_0=l=h=2$, $p_i=q_i=2$, $P(\alpha_i)=Q(\beta_{\t(i)})$, for $i=1,2$.
\end{enumerate}
\end{lemma}

\begin{proof}

For cases (1) and (2), the curve clearly has a  component of genus 0 or 1.

In case (3),  because the curve $C$ has degree $n=4$, if it is reducible
then  it has either a linear component  or a quadratic factor,
which has genus 0.  Assume that $C$ irreducible.  If there exists at
least two $i$ such that $P(\alpha_i)=Q(\beta_{\t(i)})$  then  the curve
has at least two singular points and  its genus is at most
$\displaystyle\frac{(4-1)(4-2)}2 -2=1.$ 
If  there exists only one $i$ such that $P(\alpha_i)=Q(\beta_{\t(i)})$
and $|p_i-q_{\t(i)}|=2,$ then by Lemma~\ref{quadratic},
the curve is birational  to a curve of
genus $\delta_C-1= \frac{(4-1)(4-2)}2-1-1=1.$

In case (4), the curve $C$ of degree $n$ has only one singular point
$(\alpha_1, \beta_{\t(1)}, 1)$  of multiplicity $n-1$.
Its deficiency $\displaystyle \delta_C=\frac{1}{2}(n-1)(n-2)-\frac{1}{2}(n-1)(n-2)=0.$ On the other hand, locally near the singular point, one can write
$$P(z_0)-Q(z_1) =u_1(z_0-\alpha_1)^{n}-v_1(y-\beta_{\t(1)})^{n}-v_2(y-\beta_{\t(1)})^{n-1}$$
which is easily seen to be irreducible.
Therefore, the curve $C$ has genus zero.

In case (5),   by Proposition~\ref{bkq}, the curve $C$ is irreducible.
Hence, we have the deficiency of the curve $C$ is its genus which is  $\gen_C=\delta_C=\frac{(n-1)(n-2)}2-\frac{(p_1+1)(p_1)}2=0$. 

For cases (6) and (7),  because the curve $C$ has degree $n=5$, if it is
reducible, then  it has either a linear factor  or a quadratic factor,
which therefore has genus 0. Assume that $C$ irreducible.
In case (6),  the curve has 3 singular points which are all ordinary,
so its genus is
$\frac{(5-1)(5-2)}2 - \frac{3(3-1)}2-\frac{2(2-1)}2-\frac{2(2-1)}2=1$.
In case (7),  the curve has 2 singular points which are all ordinary
of multiplicity 3. So its genus is
$\frac{(5-1)(5-2)}2 - 2.\frac{3(3-1)}2=0$. 
\end{proof}

\begin{proof}[{Proof of Theorem  \ref{th3}.}] By Lemma~\ref{dkd}, if the polynomials $P$ and $Q$ satisfy one of the cases (1) ,...,(7),
then the curve $C$ has an irreducible component of genus 0 or 1.
We now assume they do not fall into any of the cases (1) ,...,(7).
Since  $P(x)$ and $Q(x)$ are not linear polynomials, we can  assume
both $l$ and $h$ are not zero.

We will consider the following cases.

\medskip
  \noindent{\it Case 1.}\quad $l_0=0.$
\smallskip

In this case, the curve $C$ does not have any singular points. Therefore,
 it is irreducible of genus $\gen_C=\frac{1}{2}(n-1)(n-2)\ge 0$ if $n\ge 4$.

\medskip
  \noindent{\it Case 2.}\quad $l_0=1.$
\smallskip

If $l_0=1,$  and either  $p_1=1, q_{\t(1)}=3$ and  $\displaystyle\sum_{i=l_0+1}^l{p_i}=2,$ or   $p_1=3, q_{\t(1)}=1$ and  $\displaystyle\sum_{i=l_0+1}^l{q_{\t(i)}}=2,$ then $n=4$. This is the exceptional case 3. By Theorem~\ref{th1} and  Lemma~\ref{big2}~(b), we only have to consider when  $l\le  l_0+1=2$ and $ p_{j}= 1$  for all $j=2,...,l$. Similarly, $h\le 2$ and $ q_{i}= 1 $ for all $i\ne \t(1)$. 
Therefore,  the remaining cases are
\begin{align}\label{dk1}&|p_1-q_{\t(1)}|\le2, \max(l,h)\le l_0+1=2, q_{i}= 1 \text{ for all $i\ne \t(1)$,}\cr
&\text{ and } p_{j}= 1 \text{ for all $j=2,...,l$}.
\end{align}

We will consider the following sub-cases.

\medskip
\noi{\it Subcase 1.}\quad  $l=1$ and $h=1$.
\smallskip

In this case $P(x)-Q(y)=(x-\alpha_1)^{n}-(y-\beta_1)^{n}$ has linear factors, this is the exceptional case 1.

\medskip
\noi{\it Subcase 2.}\quad  $l=1$ and $h=2$  (or $l=2$ and $h=1$).
\smallskip

Since $n=m$, it follows that
$p_1=q_{\t(1)}+q_{i, i\ne \t(1)}.$ By the condition (\ref{dk1}), $q_{i}=1$  for  $i\ne \t(1)$, we have $n-1=p_1=q_{\t(1)}+1.$ This is the exceptional case 4.

Similarly,  $l=2$,$h=1$, $q_{\t(1)}=p_1+1$ and $p_2=1$ is the exceptional case 4.

\medskip
\noi{\it Subcase 3.}\quad  $l=2$ and $h=2$.
\smallskip

 By Theorem~\ref{th1}, we may assume that $|p_1-q_{\t(1)}|\le 1.$ However, by the assumption $n=m$, we have
$n-1=p_1+p_2=q_{\t(1)}+q_{i, i\ne \t(1)}$. Since $p_2=q_{i, i\ne \t(1)}=1$ by (\ref{dk1}), 
 we have $p_1=q_{\t(1)}$ and $n=p_1+2$. This is the exceptional case 5.

\medskip
\noi{\it Case 3.}\quad  $l_0\ge 2$. 
\smallskip

If $l_0=2$, $l=l_0+1=3$ and $p_{l_0+1}=1$, and $p_1=p_2=1$ then $n=4$, 
and this is the exceptional case 3. 

By Lemma \ref{big2},  we only have to consider   $|p_i-q_{\t(i)}|\le 1,$ for every $i=1,...,l_0$ and $l=l_0$, (and, similarly, $h=l_0$).

 Without loss of generality, assume that $p_1\ge p_2\ge\dots\ge p_{l_0}$. 
Take
\begin{align*}\omega&=\frac{L_{12}^{3}W(z_1,z_2)}{(z_0-\alpha_{1}z_2)^{3}(z_0-\alpha_{2}z_2)^{3}}.\end{align*}
Hence, by Lemma~\ref{regular}(iv,a), $\omega$ is regular, from which 
it follows that
$$\omega_1=z_0\omega,\ \text{ and }  \omega_2=z_1\omega$$
are two regular well-defined 1-forms. Because the curve $C$ does not have any linear components, they are $\CC$-linearly independent on every irreducible
component. However, if $p_2\ge3$ then the denominator of $\omega$ is a factor of $\displaystyle\frac{\partial F}{\partial z_0}$. So, all of the conditions in the Key Lemma are satisfied.

All together, the remaining cases are $l=h=l_0$, $|p_i-q_{\t(i)}|\le 1,$ for every $i=1,...,l_0$ and either $p_2=1$ or $p_2=2.$ We  have followings:
 
 \medskip
\noi
{\it Subcase 1.}  $p_2=1$ and  $l_0=2$.
\smallskip

Since $p_2=1$, we have  $q_{\t(2)}\le p_2+1=2$.

 If  $q_{\t(2)}=1$ then, by $n=m$,  $p_1=q_{\t(1)},$
and  the curve has degree $p_1+2$ and has two singular points
of multiplicity 2 and $p_1+1,$ all of which are ordinary.
By Proposition~\ref{bkq}(ii), the curve  has a linear factor,
which is exceptional case 1.

If  $q_{\t(2)}=2,$ then $p_1=q_{\t(1)}+1\ge 2$. If $p_1=2,$
then this is exceptional case 3. If $p_1\ge 3,$ then we consider
 \begin{align*}
&\gamma_{1,1}=\frac{L_{12}W(z_1,z_2)}{(z_0-\alpha_1z_2)^{2}(z_0-\alpha_2z_2)},\\
&\gamma_{1,2}=\frac{L_{12}^{2}W(z_1,z_2)}{(z_0-\alpha_1z_2)^{3}(z_0-\alpha_2z_2)},\end{align*}
which are regular 1-forms by Lemma~\ref{regular}, which are
$\CC$-linearly independent
because  the curve does not have any  linear components, and are
such that their denominators are  factors of
$\displaystyle\frac{\partial F}{\partial z_0}$ by  $p_1\ge 3$.
So, all of the conditions in the Key Lemma are satisfied.

 \medskip
\noi
{\it Subcase 2.}  $p_2=1$ and   $l_0=3$.
\smallskip

If  $p_1=1,$ then $p_i=q_{\t(i)}=1$ for all $i=1,2,3$ (because we assumed $p_1\ge p_2\ge...\ge p_{l_0}$). This is  the exceptional case 3.

If  $p_1\ge 2$, then either $q_{\t(1)}=p_1,$ or $p_1=q_{\t(1)}+1$ (by $l=h=l_0$). Consider two well-defined 1-forms
\begin{align*}&\gamma_{2,1}=\frac{L_{12}L_{13}W(z_1,z_2)}{(z_0-\alpha_1z_2)^{2}(z_0-\alpha_2z_2)(x-\alpha_3z_2)},\\
&\gamma_{2,2}=\begin{cases}\displaystyle\frac{L_{12}L_{23}W(z_1,z_2)}{(z_0-\alpha_1z_2)^{2}(z_0-\alpha_2z_2)(x-\alpha_3z_2)}& \text{ if $p_1=2$ and $q_1=1$}\\
\displaystyle\frac{L_{12}L_{13}^2W(z_1,z_2)}{(z_0-\alpha_1z_2)^{3}(z_0-\alpha_2z_2)(x-\alpha_3z_2)} &\text{ if } p_1\ge 3. 
\end{cases}.\end{align*}
They are regular by Lemma~\ref{regular} and nontrivial on every irreducible
component of the curve $C$ because $C$ does not have any linear components.
However, if they are $\CC$-linearly dependent, then the curve $C$
has a quadric factor, which is impossible since we can construct at least
one regular nontrivial 1-form $\gamma_{2,1}$. The condition on $p_i$
ensures that condition (i) of the Key Lemma is satisfied. Therefore,
all the conditions in the Key Lemma are satisfied.

\noindent If $p_1=q_{\t(1)}=2,$ then it is exceptional case 6.

 \medskip
\noi
{\it Subcase 3.}  $p_2=1$, $l_0\ge4.$
\smallskip

In this case, we consider 
\begin{align*}
&\gamma_{3,1}=\frac{L_{12}L_{34}W(z_1,z_2)}{(z_0-\alpha_1z_2)(z_0-\alpha_2z_2)(x-\alpha_3z_2)(x-\alpha_4z_2)},\\
&\gamma_{3,1}=\frac{L_{13}L_{24}W(z_1,z_2)}{(z_0-\alpha_1z_2)(z_0-\alpha_2z_2)(x-\alpha_3z_2)(x-\alpha_4z_2)},\end{align*}
and it is 
easy to see they satisfy all the conditions in the Key Lemma.

 \medskip
\noi
{\it Subcase 4.}  $p_2=2$.
\smallskip

In this case, $p_1\ge,2$ and using same arguments  as above, we can show
the following 1-forms  satisfy all the conditions in the Key Lemma:
\begin{align*}&\gamma_{4,1}=\frac{L_{12}^2W(z_1,z_2)}{(z_0-\alpha_1z_2)^2(z_0-\alpha_2z_2)^2},\\
&\gamma_{4,2}=\begin{cases}\displaystyle\frac{L_{12}^2L_{13}W(z_1,z_2)}{(z_0-\alpha_1z_2)^2(z_0-\alpha_2z_2)^2(z_0-\alpha_3z_2)}& \text{ if $l_0\ge 3$}\\
\displaystyle\frac{L_{12}^3W(z_1,z_2)}{(z_0-\alpha_1z_2)^3(z_0-\alpha_2z_2)^2} &\text{ if  $l_0=2$ and $p_1\ge 3$}. 
\end{cases}.\end{align*}

If $l_0=2$ and $p_1=2,$ then $n=5$ and either
$q_{\t(1)}=q_{\t(2)}=2,$ or $q_{\t(1)}=3$ and $q_{\t(2)}=1$.
When $q_{\t(1)}=3$ and $q_{\t(2)}=1$, we work similarly to the subcase 1
 when $p_2=1$, $p_1\ge3$ and $q_{\t(1)}=q_{\t(2)}=2$, which means
we can construct two regular 1-forms $\gamma_{1,1}$ and $\gamma_{1,2}$.
When $q_{\t(1)}=q_{\t(2)}=2$, it is the exceptional case 7.
\end{proof}

\end{document}